\documentclass[12pt]{amsart}
\usepackage{amsmath,amscd,amssymb}
\usepackage{latexsym}
\usepackage{dsfont}
\usepackage{color}

\def\BB{{\mathcal B}}
\def\cA{{\mathcal A}}

\def\al{\alpha}
\def\e{\varepsilon}
\def\si{\sigma}
\def\la{\lambda}

\newcommand{\NN}{\mathds{N}}
\newcommand{\TT}{\mathds{T}}
\newcommand{\ZZ}{\mathds{Z}}
\newcommand{\RR}{\mathds{R}}
\newcommand{\CC}{\mathds{C}}

\def\card{{\rm card}\,}
\def\dens{{\rm dens}\,}
\newcommand{\RRe}{\mathop{\rm Re}}

\newtheorem{theorem}{Theorem}[section]
\newtheorem{proposition}[theorem]{Proposition}
\newtheorem{lemma}[theorem]{Lemma}
\newtheorem{corollary}[theorem]{Corollary}

\theoremstyle{definition}

\newtheorem{definition}[theorem]{Definition}

\newtheorem{remark}[theorem]{Remark}

\newtheorem{example}[theorem]{Example}

\newcommand{\aveN}{\frac1N \sum_{n=1}^N}
\newcommand{\aveM}{\frac1M \sum_{n=1}^M}
\newcommand{\limaveN}{\lim_{N\to\infty}\frac1N \sum_{n=1}^N}

\begin{document}
\title{Power bounded operators and the mean ergodic theorem for subsequences}

\author{Tanja Eisner and Vladimir M\"uller}

\address{
Institute of Mathematics, University of Leipzig, Germany}
\email{eisner@math.uni-leipzig.de}

\address{Institute of Mathematics,
Czech Academy of Sciences, 
\v Zitna 25, Prague,
 Czech Republic}
\email{muller@math.cas.cz}

\thanks{This work was partially done during the second author's stay at the University of
Leipzig supported by the Alexander von Humboldt Foundation,
Germany. 
The second author was also supported by grant No. 20-31529X of GA CR and RVO:67985840.}

\maketitle

\centerline{\small \emph{Dedicated to the memory of Michael Boshernitzan}}

\begin{abstract}
Let $T$ be a power bounded Hilbert space operator without unimodular eigenvalues. 
We show that the subsequential ergodic 
 averages $N^{-1}\sum_{n=1}^N T^{a_n}$ converge in the strong operator topology for a wide range of sequences $(a_n)$, including the integer part of most of subpolynomial Hardy functions. Moreover,  we show that the weighted averages $N^{-1}\sum_{n=1}^N e^{2\pi i g(n)}T^{a_n}$ also converge for many reasonable functions $g$. In particular, we 
generalize
the polynomial mean ergodic theorem for power bounded operators due to ter Elst and the second author \cite{tEM} to real polynomials and polynomial weights.
\end{abstract}

\section{Introduction}

By the well-known mean ergodic theorem, the Ces\`aro averages $N^{-1}\sum_{n=1}^N T^n$ converge in the strong operator topology as $N\to\infty$ for any power bounded operator $T$ on a reflexive Banach space. Moreover, the limit operator is the projection onto $\ker (I-T)$ along $\overline{{\rm ran}\,(I-T)}$.

Let $(a_n)$ be a strictly increasing sequence of positive integers. The problem whether it is possible to replace the Ces\`aro averages with respect to the full sequence $(T^n)$ of all powers of $T$ by the
subsequence $(T^{a_n})$ has been studied intensely.

There are many results for unitary operators or Hilbert space contractions. The following characterization was proved in \cite{BE} and \cite{LOT}, where  SOT abbreviats the strong operator topology.

\begin{theorem}\label{thm:subseq-charac-numbers}
Let $(a_n)$ be a strictly increasing sequence of positive integers. Then the following statements are equivalent:
\begin{itemize}
\item[(i)] $(SOT)-\lim_{N\to\infty}N^{-1}\sum_{n=1}^N T^{a_n}$ exists for all Hilbert space contractions $T$;

\item[(ii)] $(SOT)-\lim_{N\to\infty}N^{-1}\sum_{n=1}^N T^{a_n}$ exists for all unitary operators $T$;

\item[(iii)] $\lim_{N\to\infty}N^{-1}\sum_{n=1}^N \la^{a_n}$ exists for all complex numbers $\la$, $|\la|=1$.
\end{itemize}
\end{theorem}

\noindent The equivalence (ii)$\Leftrightarrow$(iii) is based on the spectral theory of unitary operators, while the equivalence (i)$\Leftrightarrow$(ii) is based on the dilation theory.

Another tool for proving the mean ergodic type results for subsequences is the 
van der Corput lemma, see \cite[p.~184]{EW}.
All of these methods enable to prove the convergence of the averages
$N^{-1}\sum_{n=1}^N T^{a_n}$ for many reasonable sequences $(a_n)$ and for all contractions on Hilbert spaces, see \cite{BLRT}, \cite{BL}.

A different generalization of the mean ergodic theorem are weighted ergodic theo\-rems where one studies strong convergence of weighted averages $N^{-1}\sum_{n=1}^N c_nT^{n}$ for a given sequence of weights $(c_n)\subset\CC$. Combining weighted ergodic averages with the ergodic averages along subsequences we arrive at the mixed type of ergodic averages of the form $N^{-1}\sum_{n=1}^N c_nT^{a_n}$ with $(c_n)\subset\CC$ and $(a_n)$ being a subsequence of $\NN$. Analogously to Theorem \ref{thm:subseq-charac-numbers} one easily obtains the following characterization for convergence of such averages for contractions on Hilbert spaces, cf.~\cite{BLRT}.
\begin{theorem}\label{thm:subseq-weights-charac-numbers}
Let $(a_n)$ be a strictly increasing sequence of positive integers and $(c_n)\subset\CC$ be bounded (or, more generally, satisfy $\sup_{N\in\NN}N^{-1}\sum_{n=1}^N |c_n|<\infty$). Then the following statements are equivalent:
\begin{itemize}
\item[(i)] $(SOT)-\lim_{N\to\infty}N^{-1}\sum_{n=1}^N c_n T^{a_n}$ exists for all Hilbert space contractions $T$;

\item[(ii)] $(SOT)-\lim_{N\to\infty}N^{-1}\sum_{n=1}^N c_n T^{a_n}$ exists for all unitary operators $T$;

\item[(iii)] $\lim_{N\to\infty}N^{-1}\sum_{n=1}^N c_n \la^{a_n}$ exists for all complex numbers $\la$, $|\la|=1$.
\end{itemize}
\end{theorem}

\noindent Note that 
the most natural class of 
weights are unimodular ones, i.e., of the form $c_n=e^{2\pi i g(n)}$, $n\in\NN$, for some $g:\NN\to\RR$. 
\smallskip 

All the above mentioned methods do not work for power bounded operators, which form a natural class from the point of view of the mean ergodic theorem.
It is worth to point out that power bounded operators on Hilbert spaces have very different properties from contractions. The study of power bounded operators and their relations to contractions has a long history. For main results see \cite{SzN}, \cite{Fo}, \cite{H}, \cite{P}.


In \cite{tEM}, the strong convergence of the averages
$\frac{1}{N}\sum_{n=1}^N T^{p(n)}$ 
was proved for all power bounded Hilbert space operators $T$ and all polynomials $p$ 
satisfying $p(\NN)\subset\NN$, which was the only known mean ergodic theorem along a non-trivial subsequence for power bounded operators. The present paper is an attempt to fill this gap, also regarding weighted ergodic theorems along subsequences.   
%

We extend the results of \cite{tEM} 
to a wide range of sequences $(a_n)$ of subpolynomial growth, e.g. $a_n=\Bigl[\sum_{j=0}^k c_jn^{\al_j}\Bigr]$, where  $c_0,\dots,c_k,\al_0,\dots,\al_k\in\RR$, $c_0>0$, $\al_0>\max\{0,\al_1,\dots,\al_k\}$ and $[\cdot]$ denotes the integer part, or $a_n=[n^\al\ln^\beta n]$, $\al>0, \al\notin\NN$, $\beta\in\RR$, see Theorem \ref{thm:subseq}.
Moreover, we prove also the strong convergence of the Ces\`aro averages 
$$
\aveN e^{2\pi i g(n)}T^{a_n}
$$ 
for many natural functions $g$ including  real polynomials, see Theorems \ref{thm:main-weighted}, \ref{thm:weighted-pol} and Corollary \ref{cor:weighted}. Our main examples will be again (large classes of) Hardy functions. In particular, we generalize the result of ter Elst, M\"uller [tEM] to real polynomials and polynomial weights, see Corollary \ref{cor:pol}.

Our investigations are inspired by ergodic theory where subsequential and weighted ergodic theorems have been active areas of research for many years with connections to other areas of mathematics such as harmonic analysis and number theory, see, e.g., \cite[Chapter 21]{EFHN}, \cite{Bou},\cite{W},\cite{N},\cite{RW},\cite{BM},\cite{A},\cite{L},\cite{EK},\cite{S},
\cite{GT}.


\medskip 

Let $T$ be a power bounded operator on a reflexive Banach space $X$.
By the Jacobs-Glicksberg-deLeeuw theorem, see, e.g., \cite[Thm.~I.1.5]{E},
there is a decomposition $X=X_1\oplus X_2$, where 
\begin{eqnarray*}
X_1&=&\overline{\text{lin}}\{x\in X:\, Tx=\lambda x\text{ for some }\lambda\in\TT\},\\
X_2&=&\left\{x\in X:\, 0\in\overline{\{T^nx,n\in\NN\}}^{\text{weak}}\right\},
\end{eqnarray*}
$\TT$ denoting the unit circle.
It is 
easy to see that the strong convergence of the Ces\`aro averages $N^{-1}\sum_{n=1}^N (T|_{X_1})^{a_n}$ is equivalent to the convergence of $N^{-1}\sum_{n=1}^N \la^{a_n}$ for all $\la\in\si_p(T)\cap\TT$. Moreover, an analogous characterization holds for strong convergence of weighted averages  $N^{-1}\sum_{n=1}^Ne^{2\pi i g(n)} (T|_{X_1})^{a_n}$.
So strong convergence of subsequential/weighted ergodic averages
of the operator 
$T|_{X_1}$ restricts to the same condition 
as in Theorem \ref{thm:subseq-charac-numbers}(iii) or Theorem \ref{thm:subseq-weights-charac-numbers}(iii), respectively, for all unimodular eigenvalues $\lambda$ of $T$. 
In this paper 
we mostly
concentrate on the operator $T|_{X_2}$. 
Thus we 
assume that our power bounded operator $T$ has no peripheral point spectrum, $\si_p(T)\cap\TT=\emptyset$.

\textbf{Acknowledgment.} We thank the referee for 
inspiring comments. 

\section{Functions of subpolynomial growth}

Denote by $B$ the set of all germs at $+\infty$ of continuous real functions of  real variable $t$.
So the elements of $B$ are continuous functions defined on an interval $[t_0,\infty)$; we identify two such functions if they are equal for all $t$ large enough.

Let $f,g\in B$. We write $f\ll g$ if $f(t)<g(t)$ for all $t$ large enough.

\begin{definition}\label{defPm}
Let $m\in\NN$. We say that a function $f\in B$ is of class $P_m$ if  
$f$ has 
continuous derivatives $f', f'',\dots,f^{(m)}$, $f,f',\dots, f^{(m)}\gg 0$, 
$$
\limsup_{t\to\infty}\frac{f^{(m-1)}(t)}{t\,f^{(m)}(t)}<\infty
$$ 
and
$$
\limsup_{t\to\infty}\sup\Bigl\{\frac{f^{(m)}(s)}{f^{(m)}(t)}: s\ge t\Bigr\}<\infty.
$$
\end{definition}

Note that the last condition is satisfied if either
$f^{(m)}$ is decreasing,
or $\lim_{t\to\infty}f^{(m)}(t)$ exists and is positive.

Typical functions satisfying conditions of Definition \ref{defPm} are real polynomials of degree $m$ with positive leading coefficient, $f(t)=t^\al\quad(m-1<\al\le m)$ or $f(t)=t^\al\ln^\beta t\quad(m-1<\al<m,\beta\in\RR)$. For more examples see Section \ref{sec:hardy} below.

The following lemmas describe properties of functions of class $P_1$.

\begin{lemma}\label{sublinear}
Let $f$ be a function of class $P_1$.
Then:
\begin{itemize}
\item[(i)]
$\lim_{t\to\infty}f(t)=\infty$;

\item[(ii)]
$\limsup_{t\to\infty}\frac{f(t)}{t}<\infty$;

\item[(iii)]
$\limsup_{t\to\infty}\frac{f(2t)}{f(t)}<\infty$;

\item[(iv)]
$\liminf_{t\to\infty}\frac{f(t)}{tf'(t)}>0$.
\end{itemize}
\end{lemma}

\begin{proof}
Let $c,c'>0$ and $t_0$ satisfy that $f(t)>0$ and $f'(t)>0$ for $t\ge t_0$,
$\frac{f(t)}{tf'(t)}\le c\quad(t\ge t_0)$ and $f'(s)\le c'f'(t)\quad(t_0\le t\le s)$.

(i)
For $t\ge t_0$ we have
$$
f(t)= f(t_0)+\int_{t_0}^t f'(s) ds\ge
\int_{t_0}^t \frac{f(s)}{cs}ds\ge
\frac{f(t_0)}{c}\int_{t_0}^t\frac{ds}{s} =\frac{f(t_0)}{c}\ln\frac{t}{t_0}\to\infty
$$
as $t\to\infty$.
\smallskip

(ii) For $t\ge t_0$ we have
$$
\frac{f(t)}{t}\le cf'(t)\le cc'f'(t_0).
$$
So $\limsup_{t\to\infty}\frac{f(t)}{t}<\infty$.
\smallskip

(iii)
For $t\ge t_0$ we have
$$
f(t)=f(t_0)+\int_{t_0}^t f'(s)ds
$$
and
$$
f(2t)=f(2t_0)+\int_{2t_0}^{2t} f'(u)du
=f(2t_0)+2\int_{t_0}^{t} f'(2s)ds\le
f(2t_0)+2c'\int_{t_0}^t f'(s)ds
$$
$$
=
f(2t_0)+2c'(f(t)-f(t_0)).
$$
Hence by (i),
$$
\frac{f(2t)}{f(t)}\le
\frac{f(2t_0)}{f(t)}+2c'-\frac{2c'f(t_0)}{f(t)}\to 2c'<\infty
$$
as $t\to\infty$.
\smallskip

(iv)
For $t\ge t_0$ we have
$$
f(t)=f(t_0)+\int_{t_0}^t f'(s)ds\ge
f(t_0)+\frac{1}{c'}(t-t_0)f'(t).
$$
So
$$
tf'(t)\le
c'\bigl(f(t)-f(t_0)\bigr)+t_0f'(t)
$$
and by (i),
$$
\frac{tf'(t)}{f(t)}\le c'-\frac{c'f(t_0)}{f(t)}+t_0\frac{f'(t_0)}{f(t)}\to c'<\infty.
$$
Hence
$\liminf_{t\to\infty}\frac{f(t)}{tf'(t)}\ge c^{'-1}>0$.

\end{proof}

\begin{lemma}\label{sublinear2}
Let $f$ be a function of class $P_1$.
For $k\in\NN$ let 
$$
b_k=\min\bigl\{n\in\NN: f \hbox{ is defined and non-decreasing for } t\ge n \hbox{ and }f(n)\ge k\bigr\}.
$$
Then:
\begin{itemize}
\item[(i)]
$\lim_{k\to\infty}b_k=\infty$;

\item[(ii)]
$\limsup_{k\in\NN}\frac{k}{b_k}<\infty$;

\item[(iii)]
$$
\sup_{k\in\NN}\frac{k(b_{k+1}-b_k)}{b_k} <\infty;
$$
Consequently, $\lim_{k\to\infty}\frac{b_{k+1}}{b_k}=1.$

\item[(iv)]
$\limsup_{j\to\infty}\sup_{k\ge j} \frac{d_j}{d_k+1}<\infty$, where $d_j=b_{j+1}-b_j$. So
$\sup\Bigl\{\frac{d_j}{d_k+1}: j\le k\Bigr\}<\infty$.
\end{itemize}
\end{lemma}

\begin{proof}
Let $c,c'>0$ and $t_0$ satisfy that $f(t)>0$ and $f'(t)>0$ for $t\ge t_0$,
$\frac{f(t)}{tf'(t)}\le c\quad(t\ge t_0)$ and $f'(s)\le c'f'(t)\quad(t_0\le t\le s)$.
\smallskip

(i) follows from Lemma \ref{sublinear} (i).
\smallskip

(ii) We have $f(b_k-1)< k$ for $b_k\geq t_0+1$. So
$$
\limsup_{k\to\infty}\frac{k}{b_k}=
\limsup_{k\to\infty}\frac{k}{b_k-1}\le
\limsup_{t\to\infty}\frac{f(t)}{t}<\infty
$$
by Lemma \ref{sublinear} (ii).
\smallskip

(iii) 
For $k$ large enough and $t>b_k$ we have 
\begin{equation}\label{(1)}
f(t)-f(b_k)=\int_{b_k}^t f'(s) ds\ge
\int_{b_k}^t \frac{f(s)}{cs} ds\ge
\frac{f(b_k)}{c}\int_{b_k}^t\frac{ds}{s}\ge
\frac{k}{c}\ln\frac{t}{b_k}.
\end{equation}
Let $t_1=b_k\cdot\exp\frac{c}{k}$. Then $t_1> b_k$, $\frac{k}{c}\ln\frac{t_1}{b_k}=1$ and
$f(t_1)\ge f(b_k)+1\ge k+1$ by \eqref{(1)}. So $b_{k+1}\le t_1+1$ and
$b_{k+1}-b_k\le b_k\exp\frac{c}{k}-b_k+1$.
We have
$$
\limsup_{k\to\infty} \frac{k(b_{k+1}-b_k)} {b_k}\le
\limsup_{k\to\infty} \Bigl( k\bigl(\exp\frac{c}{k}-1\bigr) + \frac{k}{b_k}\Bigr) <\infty
$$
since
$\lim_{k\to\infty}  k\bigl(\exp\frac{c}{k}-1\bigr)=c <\infty$.
In particular,
$$
\lim_{k\to\infty} \frac{b_{k+1}-b_k}{b_k}=0
$$
and
$$
\lim_{k\to\infty} \frac{b_{k+1}}{b_k}=1.
$$

\smallskip
(iv)
Let $j$ be large enough and $k\ge j$. Write for short $d_j=b_{j+1}-b_j$.
If either $k=j$ or $d_j\le 1$ then clearly
$\frac{d_j}{d_k+1}\le 1$. 
So we may assume that $k>j$ and $d_j\ge 2$.

We have $f(b_{k+1})\ge k+1$ and $f(b_k-1)<k$. So 
\begin{eqnarray*}
1&<& f(b_{k+1})-f(b_k-1)=
\int_{b_k-1}^{b_{k+1}} f'(t) dt\le
(d_{k}+1)\cdot\max\{f'(t):b_k-1<t<b_{k+1}\}\\
&\le&
c'(d_{k}+1)\cdot\min\{f'(t):b_j<t<b_{j+1}-1\}
\le
\frac{c'(d_k+1)}{d_j-1}\int_{b_j}^{b_{j+1}-1} f'(t) dt<
\frac{c'(d_k+1)}{d_j-1}
\end{eqnarray*}
since $\int_{b_j}^{b_{j+1}-1} f'(t) dt= f(b_{j+1}-1)-f(b_j)<1$.
Hence
$d_j\le c'd_k+c'+1$ and 
$$
\sup\left\{\frac{d_j}{d_k+1}:j\le k\right\}<\infty.
$$ 
\end{proof}

In the following we will consider more general functions - bounded perturbations of functions satisfying conditions of Lemma \ref{sublinear}.

Let $T$ be a bounded linear operator on a Banach space $X$. Let $f\in P_1$ and let $(h_n)$ be a bounded integer-valued sequence. Denote by $[\cdot]$ the integer part. In the sequence $([f(n)]+h_n)$ there may be a finite number of negative terms, or even the function $f$ is not defined, and so the power $T^{[f(n)]+h_n}$ is not defined. However, the convergence of the Cesaro averages does not depend on a finite number of terms. To avoid technical difficulties, we use the convention that $T^{[f(n)]+h_n}=I$ (the identity operator on $X$) if the exponent is negative or not defined.

\begin{definition}
For $a,b\in\NN$, $a\le b$ denote by $[a,b]$ the interval $\{n\in\NN: a\le n\le b\}$.

Let $A\subset\NN$. We say that $A$ has \emph{density} $\dens A$ if
$$
\lim_{n\to\infty}\frac{\card(A\cap[1,n])}{n}=\dens A.
$$
\end{definition}

\begin{example}\label{ex:rtseq-density}
A large class of examples for sets with positive density is provided by ergodic theory. Let $(X,\mu)$ be a probability space and $T:X\to X$ be a measure preserving transformation, i.e., $\mu(T^{-1}(B))=\mu(B)$ holds for every measurable $B\subset X$. For $B\subset X$  with $\mu(B)>0$ and $x\in X$, the set 
$$
A:=\{n:\, T^nx\in B\}
$$
is called the set of all return times of $x$ to $B$. If $(X,\mu, T)$ is ergodic, i.e., if every $T$-invariant set has either full or zero measure, Birkhoff's ergodic theorem implies that $\dens A=\mu(B)$ for almost all $x\in X$. Moreover, if $(X,T)$ is uniquely ergodic, i.e., $X$ is a compact space, $T$ is continuous and $\mu$ is the unique $T$-invariant measure, then $d(A)=\mu(B)$ holds for all initial values $x\in X$.
\end{example}

Recall that a linear operator $T$ on a Banach space $X$ is called to have \emph{relatively weakly compact orbits} if for every $x\in X$ the set $\{T^nx:\,n\in\NN_{0}\}$ is relatively compact with respect to the weak topology,  where $\NN_0:=\NN\cup \{0\}$. Note that every such operator is automatically power bounded.  Moreover, for reflexive Banach spaces every power bounded operator has relatively weakly compact orbits. For more examples of operators with relatively weakly compact orbits see, e.g., \cite[Example I.1.7]{E} and 
 \cite{KL}.

\begin{theorem}\label{caseP1}
Let $T$ have relatively weakly compact orbits  on a Banach space $X$,  $\si_p(T)\cap\TT=\emptyset$, and $x\in X$. 
Let $f\in P_1$. Let $(h_n)$ be a bounded integer-valued sequence and let $A\subset\NN$ be a subset of positive density.
Then
$$
\lim_{N\to\infty}\sup_{x^*\in X^*, \|x^*\|=1}\frac{1}{\card(A\cap[1,N])}\sum_{n\in A\cap[1,N]}|\langle T^{[f(n)]+h_n}x,x^*\rangle|=0.\eqno(2)
$$
In particular, 
$$
\lim_{N\to\infty}\frac{1}{\card(A\cap[1,N])}\sum_{n\in A\cap[1,N]} T^{[f(n)]+h_n}x=0.\eqno(3)
$$ 
\end{theorem}

\begin{proof}
Without loss of generality we may assume that $f$ is continuous, positive and non-decreasing on $[0,\infty)$. Let $b_k=\min\{n\in\NN: f(n)\ge k\}$ be the numbers considered in Lemma \ref{sublinear2}. Let $M:=\sup\{\|T^n\|:n\in\NN_0\}$. Without loss of generality we may assume that $\|x\|=1$.

Let $c'>0$ satisfy
$\frac{f'(s)}{f'(t)}\le c'$ for all $t$ large enough and $s\ge t$. Let $d:=\dens A$.

Let $K$ be sufficiently large, $b_K\le N<b_{K+1}$ and $x^*\in X^*$, $\|x^*\|=1$. Then
$$
\frac{1}{\card(A\cap[1,N])}\sum_{n\in A\cap[1,N]} |\langle T^{[f(n)]+h_n}x, x^*\rangle|
$$
$$
\le
\frac{1}{\card(A\cap[1,b_K-1])}\sum_{n\in A\cap[1, b_{K}-1]} |\langle T^{[f(n)]+h_n}x,x^*\rangle|
$$
$$
+ \frac{1}{\card(A\cap[1,N])}\sum_{n\in A\cap[b_K,N]} |\langle T^{[f(n)]+h_n}x,x^*\rangle|
$$
where
$$
\lim_{N\to\infty}\frac{1}{\card(A\cap[1,N])}\sum_{n\in A\cap[b_K,N]} |\langle T^{[f(n)]+h_n}x,x^*\rangle|
$$
$$
\le \lim_{N\to\infty}\frac{(b_{K+1}-b_K)M}{\card(A\cap[1,N])}=
\lim_{N\to \infty}\frac{(b_{K+1}-b_K)M}{Nd}=0
$$
uniformly in $x^*$, 
since $\lim_{N\to \infty}\frac{(b_{K+1}-b_K)}{b_K}=0$ by Lemma \ref{sublinear2}\ (iii) and $b_K\le N$.

Hence it is sufficient to show that
$$
\lim_{K\to\infty}\sup_{x^*\in X^*, \|x^*\|=1}\frac{1}{\card(A\cap[1,b_K-1])}\sum_{n\in A\cap[1,b_K-1]} |\langle T^{[f(n)]+h_n}x,x^*\rangle|=0,
$$
which can be rewritten as
$$
\lim_{K\to\infty}\sup_{x^*\in X^*, \|x^*\|=1}\frac{1}{\card(A\cap[1,b_K-1])}\sum_{k=0}^{K-1}\sum_{n\in A\cap[b_k,b_{k+1}-1]} |\langle T^{k+h_n}x,x^*\rangle|=0.
$$
Let $r:=\max_n |h_n|$.

Fix $\e>0$.
By 
the Jacobs-Glicksberg-deLeeuw decomposition and a result of Jones and Lin,
see, e.g., \cite[Thm. II.4.8 and Rem. II.4.5]{E},
there exists $K_0$ such that for all $K\ge K_0$ and $x^*\in X^*$, $\|x^*\|=1$ we have
\begin{equation}\label{eq:conv-zero}
\frac{1}{K}\sum_{k=1}^K|\langle T^kx,x^*\rangle|<\e^2.
\end{equation}
We may also assume that $\card (A\cap[1,b_K-1])\ge \frac{b_K d}{2}$ for $K\ge K_0$.
Let $K\ge K_0$ and let $x^*\in X^*$, $\|x^*\|=1$ be fixed. Let
$$
L:=\{k: 1\le k\le K: |\langle T^kx,x^*\rangle|\ge\e\}.
$$
Then $\card L\le\e K$ by (\ref{eq:conv-zero}). Let
$$
\tilde L:=\{k: 1\le k\le K: \hbox{ there exists } k'\in L, k-r\le k' \le k+r\}.
$$
Then $\card \tilde L\le (2r+1)\e K$.

For $K$ large 
we have
\begin{eqnarray*}
& &\frac{1}{\card(A\cap[1,b_K-1])}\sum_{k=0}^{K-1}\ \sum_{n\in A\cap[b_k,b_{k+1}-1]}|\langle T^{k+h_n}x,x^*\rangle|\\
& &\qquad =
\frac{1}{\card(A\cap[1,b_K-1])}\sum_{k\in \tilde L\cap [0,K-1]}\ \sum_{n\in A\cap[b_k,b_{k+1}-1]}|\langle T^{k+h_n}x,x^*\rangle|
\\
& &\qquad+
\frac{1}{\card(A\cap[1,b_K-1])}\sum_{k\in[0,K-1]\setminus \tilde L}\ \sum_{n\in A\cap[b_k,b_{k+1}-1]}|\langle T^{k+h_n}x,x^*\rangle|
\\
& &\qquad \le \frac{2}{d b_K}\e(2r+1)K\max\{d_j:j\le K-1\}M+\e
\\
& &\qquad \le
2\e(2r+1)M\frac{K(d_K+1)}{b_k}dc''+\e\le c'''\e,
\end{eqnarray*}
where $c''$ and $c'''$ are constants, see Lemma \ref{sublinear2} (iv) and (iii).

Since $\e>0$ was arbitrary, we have
$$
\lim_{K\to\infty}\sup_{x^*\in X^*, \|x^*\|=1}\frac{1}{\card (A\cap[1,b_K-1])}\sum_{n\in A\cap[1,b_k-1]}|\langle T^{[f(n)]+h_n}x,x^*\rangle|=0
$$
and
$$
\lim_{N\to\infty}\sup_{x^*\in X^*, \|x^*\|=1}\frac{1}{\card(A\cap[1,N])}\sum_{n\in A\cap[1,N]}|\langle T^{[f(n)]+h_n}x,x^*\rangle|=0.
$$
In particular,
$$
\lim_{N\to\infty}\Bigl\|\frac{1}{\card(A\cap[1,N])}\sum_{n\in A\cap[1,N]} T^{[f(n)]+h_n}x\Bigr\|
$$
$$
=
\lim_{N\to\infty}\sup_{\|x^*\|=1}\frac{1}{\card(A\cap[1,N])}\Bigl|\sum_{n\in A\cap[1,N]}\langle T^{[f(n)]+h_n}x,x^*\rangle\Bigr|
$$
$$
\le
\lim_{N\to\infty}\sup_{\|x^*\|=1}\frac{1}{\card(A\cap[1,N])}\sum_{n\in A\cap[1,N]}|\langle T^{[f(n)]+h_n}x,x^*\rangle|=0.
$$
\end{proof}

\begin{remark}
By Jones, Lin \cite{JL1,JL2}, the following assertions are equivalent for power bounded operators $T$ on a Banach space $X$ which does not contain a copy of $l^1$:
\begin{itemize}
\item[(i)] $T^*$ has no eigenvalues on $\TT$.
\item[(ii)] $\displaystyle\lim_{N\to\infty}\sup_{x^*\in X^*,\|x^*\|=1}\aveN |\langle T^nx,x^*\rangle|=0$ for every $x\in X$.
\end{itemize}
Thus one can replace the conditions that $T$ has relatively weakly compact orbits on a Banach space $X$ and $\sigma_p(T)\cap \TT=\emptyset$ in Theorem \ref{caseP1} by the conditions that $T$ is a power bounded operator on a Banach space not containing a copy of $l^1$ and $\sigma_p(T^{*})\cap\TT= \emptyset$.
\end{remark}

Functions of subpolynomial growth will be treated inductively.  Recall that a function $g\in B$ has \emph{subpolynomial growth} if $|g(t)|\ll t^n$ for some $n\in\NN$.

\begin{lemma}\label{L2.8}
Let $m\ge 2$ and $f\in P_m$. Then:
\begin{itemize}
\item[(i)] $f'\in P_{m-1}$;

\item[(ii)] $\lim_{t\to\infty}f(t)=\infty$;

\item[(iii)] $\limsup_{t\to\infty}\frac{f(2t)}{f(t)}<\infty$;

\item[(iv)] $\limsup_{t\to\infty}\frac{tf^{(k)}(t)}{f^{(k-1)}(t)}<\infty$ for all $k=1,\dots,m$;

\item[(v)] for each $r>0$,
the function $g_r(t):=f(t+r)-f(t)$ belongs to $P_{m-1}$.
\end{itemize}

\end{lemma}

\begin{proof}
Let $c,c'>0$ and $t_0$ satisfy that $f$ has continuous and positive derivatives of orders $\le m$ for $t\ge t_0$,
$\frac{f^{(m-1)}(t)}{tf^{(m)}(t)}\le c\quad(t\ge t_0)$ and $f^{(m)}(s)\le c'f^{(m)}(t)\quad(t_0\le t\le s)$.

(i) follows from the definition.
\smallskip

(ii) For $m=1$ this was proved in Lemma \ref{sublinear} (i).

Let $m\ge 2$ and suppose that the statement was proved for $m-1$. Let $f$ be a function of class $P_m$. Then $f'$ is of class $P_{m-1}$. By the induction assumption, $\lim_{t\to\infty}f'(t)=\infty$. Hence $\lim_{t\to\infty} f(t)=\infty$.
\smallskip

(iii)  
For $m=1$ this was proved in Lemma \ref{sublinear} (iii). Let $m\ge 2$ and suppose that the statement is true for $m-1$.

For $t\ge t_0$ we have
$$
f(t)=f(t_0)+\int_{t_0}^t f'(s) ds
$$
and
$$
f(2t)=f(2t_0)+\int_{2t_0}^{2t} f'(s) ds=
f(2t_0)+2\int_{t_0}^t f'(2u) du.
$$
Thus 
$$
\frac{f(2t)}{f(t)}\le
\max\Bigl\{ \frac{f(2t_0)}{f(t_0)}, 2\sup\Bigl\{\frac{f'(2s)}{f'(s)}:s\ge t_0\Bigr\}\Bigr\}.
$$
Since $f'\in P_{m-1}$, we have 
$\limsup{\frac{f(2t)}{f(t)}}<\infty$ by the induction assumption.
\smallskip

(iv) 
For $m=1$ this was proved in Lemma \ref{sublinear} (iv).

Let $m\ge 2$. If $2\le k\le m$ then the statement follows by the induction assumption since $f'\in P_{m-1}$.

Let $m\ge 2$ and $k=1$. Then
$$
\limsup_{t\to\infty}\frac{tf'(t)}{f(t)}\le
\limsup_{t\to\infty}\frac{tf''(t)+f'(t)}{f'(t)}=
1+\limsup_{t\to\infty}\frac{tf''(t)}{f'(t)}<\infty
$$ by the L'Hospital rule and the induction assumption.
\smallskip

(v)
Let $r>0$ 
and $g_r(t):=f(t+r)-f(t)$. Clearly $g_r$ has continuous positive derivatives of order $\le m-1$.

For $t\ge t_0$ we have
$$
\frac{g_r^{(m-2)}(t)}{t\,g_r^{(m-1)}(t)}=
\frac{f^{(m-2)}(t+r)-f^{(m-2)}(t)}
{t\,( f^{(m-1)}(t+r)-f^{(m-1)}(t))}=
\frac{f^{(m-1)}(\xi)}{t\,f^{(m)}(\xi')}
$$
for some $\xi,\xi'\in(t,t+r)$. If $\xi\le\xi'$ then 
$$
\frac{f^{(m-1)}(\xi)}{t\,f^{(m)}(\xi')}\le
\frac{f^{(m-1)}(\xi')}{\xi'\,f^{(m)}(\xi')}\cdot\frac{\xi'}{t}\le c\cdot \frac{t+r}{t}.
$$
If $\xi>\xi'$ then
$$
\frac{f^{(m-1)}(\xi)}{t\,f^{(m)}(\xi')}=
\frac{f^{(m-1)}(\xi)}{\xi \,f^{(m)}(\xi)}\cdot
\frac{\xi}{t}\cdot\frac{f^{(m)}(\xi)}{f^{(m)}(\xi')}\le
c\cdot \frac{t+r}{t}\cdot c'.
$$
Hence
$$
\limsup_{t\to\infty}\frac{g_r^{(m-2)}(t)}{t\,g_r^{(m-1)}(t)}<\infty.
$$

Let $t_0\le t<s<\infty$. If $t+r<s$ then
$$
\frac{g_r^{(m-1)}(s)}{g^{(m-1)}(t)}=
\frac{{f^{(m-1)}(s+r)}-{f^{(m-1)}(s)}}{{f^{(m-1)}(t+r)}-{f^{(m-1)}(t)}}=
\frac{f^{(m)}(\xi)}{f^{(m)}(\xi')}
$$
for some $\xi\in(s,s+r)$ and $\xi'\in(t,t+r)$. So $\xi'\le \xi$ and the above fraction is bounded by $c'$.

Let $t<s<t+r<s+r$. Then
$$
\frac{g_r^{(m-1)}(s)}{g_r^{(m-1)}(t)}=
\frac{{f^{(m-1)}(s+r)}-{f^{(m-1)}(s)}}{{f^{(m-1)}(t+r)}-{f^{(m-1)}(t)}}
$$
$$
=
\frac{\bigl({f^{(m-1)}(s+r)}-{f^{(m-1)}(t+r)}\bigr)
+\bigl({f^{(m-1)}(t+r)}-{f^{(m-1)}(s)}\bigr)}
{\bigl({f^{(m-1)}(s)}-{f^{(m-1)}(t)}\bigr)
+\bigl({f^{(m-1)}(t+r)}-{f^{(m-1)}(s)}\bigr)}
$$
$$
\le\max\Bigl\{1,\frac{f^{(m-1)}(s+r)-f^{(m-1)}(t+r)}{f^{(m-1)}(s)-f^{(m-1)}(t)}\Bigr\}=
\max\Bigl\{1,\frac{f^{(m)}(\xi)}{f^{(m)}(\xi')}\Bigr\}
$$
for some $\xi\in(t+r,s+r)$ and $\xi'\in(t,s)$. Thus $\xi'<\xi$ and $ {f^{(m)}(\xi)}\le c' {f^{(m)}(\xi')}$.
So 
$\frac{g_r^{(m-1)}(s)}{g_r^{(m-1)}(t)}\le\max\{c',1\}$.

Thus
$$
\limsup_{t\to\infty}\sup\left\{\frac{g_r^{(m-1)}(s)}{g_r^{(m-1)}(t)}: s\ge t\right\}<\infty.
$$

Hence the functions $g_r$ have property $P_{m-1}$.
\end{proof}

\section{Power bounded operators on Hilbert spaces}

The following van der Corput type result for power bounded operators on Hilbert spaces is a variation of ter Elst, M\"uller, \cite[Thm. 2.1]{tEM}.

\begin{theorem} \label{tvdc201}
Let $T$ be a power bounded operator acting on a Hilbert space $H$ and let $x\in H$.
Let $(a_n)_{n=1}^\infty$ be a strictly increasing sequence of positive integers 
such that 
$\sup\bigl\{ \frac{a_{2n}}{a_n}:n\in\NN\bigr\}<\infty$ and 
$\lim_{n\to\infty}\frac{a_{n+1}}{a_n}=1$.
Suppose that
\[
\lim_{N\to\infty}N^{-1}\sum_{j=1}^N T^{a_{j+k}-a_j}x=0
\]
for all $k\in\NN$.
Then
\[
\lim_{N\to\infty}N^{-1}\sum_{j=1}^N T^{a_j}x=0.
\]
\end{theorem}

\begin{proof}
Without loss of generality we may assume that $\|x\|=1$.
Let 
\[
M:=\sup\{\|T^n\|: n \ge 0 \}.
\]
Suppose on the contrary that there exists an $\eta>0$ such that
\[
\limsup_{N\to\infty}N^{-1}\Bigl\|\sum_{j=1}^NT^{a_j}x\Bigr\|>\eta.
\]
Fix $k\in\NN$ such that $k>\frac{24 M^4c}{\eta^2}$, where 
$c= \sup\bigl\{\frac{a_{2n}}{a_n}:n\in\NN\bigr\}.$

By the assumptions, $\lim_{n\to\infty}\frac{a_{n+1}}{a_n}=1$, and so
$\lim_{n\to\infty}\frac{a_{n+k}}{a_n}=1$. Thus
$$
\lim_{n\to\infty}\frac{a_{n+k}-a_n}{a_n}=0.
$$

Let $N_0\in \NN$ be such that $N_0\ge\max\{\frac{2kM}{\eta},4k\}$, 
\[
\frac{4M(a_{N+k}-a_{N})}{a_N}< k^{-1}
\]
for all $N\ge N_0$ and
\begin{equation}
N^{-1}\Bigl\|\sum_{j=0}^N T^{a_{j+l}-a_j}x\Bigr\|< k^{-1}
\label{eSvdc1;1}
\end{equation}
for all $N\ge N_0$ and $l \in \{ 1,2,\dots,k-1 \} $.

We need a lemma.

\begin{lemma} \label{lvdc202}
There exists an $N\ge N_0$ such that
\[
N^{-1}\Bigl\|\sum_{j=N+1}^{2N} T^{a_j}x\Bigr\|>\eta.
\]
\end{lemma}
\begin{proof}
Fix $\eta_1$ such that 
$\eta<\eta_1<\limsup_{N'\to\infty}N'^{-1}\Bigl\|\sum_{j=1}^{N'} T^{a_j}x\Bigr\|$.
Let $v\in\NN$ be such that $\frac{M}{2^v}<\frac{\eta_1-\eta}{2}$.
There exists an $N_2\ge 4^v N_0$ such that
\[
N_2^{-1}\Bigl\|\sum_{j=1}^{N_2} T^{a_j}x\Bigr\|> \eta_1.
\]
Write $N_2=2^v\cdot N_1 +z$, where $0\le z <2^v$.
Then $N_1\ge N_0$.
Suppose on the contrary that $N^{-1}\Bigl\|\sum_{j=N+1}^{2N}T^{a_j}x\Bigr\|\le \eta$ 
for all $N\ge N_0$.
Then in particular,
\[
\frac{1}{2^{i}N_1}\Bigl\|\sum_{j=2^{i}N_1+1}^{2^{i+1}N_1}T^{a_j}x\Bigr\|
\le \eta
\]
for all $i \in \{ 0,1,\dots, v-1 \} $.
So
\begin{eqnarray*}
N_2^{-1}\Bigl\|\sum_{j=1}^{N_2}T^{a_j}x\Bigr\|
&\le& N_2^{-1}\Bigl(
\Bigl\|\sum_{j=1}^{N_1}T^{a_j}x\Bigr\|+
\Bigl\|\sum_{j=N_1+1}^{2N_1}T^{a_j}x\Bigr\|+\cdots  \\*
&&\hskip 1.5cm \cdots +
\Bigl\|\sum_{j=2^{v-1}N_1+1}^{2^vN_1}T^{a_j}x\Bigr\|
   + \Bigl\|\sum_{j=2^vN_1+1}^{N_2}T^{a_j}x\Bigr\|\Bigr) \\
&\le& N_2^{-1}\Bigl(N_1M+\eta N_1+2\eta N_1+\cdots+ 2^{v-1}\eta N_1+2^vM\Bigr)  \\
&\le&  \frac{N_1M}{N_2}+\frac{2^v\eta N_1}{N_2}+\frac{2^vM}{N_2}
\le \eta +\frac{2M}{2^v}\le \eta_1,
\end{eqnarray*}
which is a contradiction.
\end{proof}

\noindent{\bf Continuation of the proof of Theorem~\ref{tvdc201}.}
Fix $N\ge N_0$ as in Lemma~\ref{lvdc202}.
Write for short $x_j=T^jx$ for all $j\in\NN$.
For all $r \in \{ 1,\ldots,a_N \} $ and $s \in \{ N+1,\ldots,2N \} $ write
\[
u_{r,s}
:=x_r+x_{r+a_{s+1}-a_s}+\cdots+ x_{r+a_{s+k-1}-a_s}.
\]
Then
\[
T^{a_s-r}u_{r,s}=x_{a_s}+x_{a_{s+1}}+\cdots+ x_{a_{s+k-1}}.
\]
Consider
\[
A:=\frac{1}{a_N N}\sum_{r=1}^{a_N}\sum_{s=N+1}^{2N}\|u_{r,s}\|^2.
\]
We will estimate $A$ from above and from below to obtain a contradiction.

First we consider a lower bound.
Clearly
\begin{eqnarray*}
A
&\ge& \frac{1}{M^2a_NN}\sum_{r=1}^{a_N}\sum_{s=N+1}^{2N}
    \|x_{a_s}+x_{a_{s+1}}+\cdots+x_{a_{s+k-1}}\|^2\cr
&=&\frac{1}{M^2N}\sum_{s=N+1}^{2N}\|x_{a_s}+x_{a_{s+1}}+\cdots+x_{a_{s+k-1}}\|^2.
\end{eqnarray*}
The Cauchy--Schwarz inequality and the triangular inequality then give
\begin{eqnarray*}
A&\ge&
\frac{1}{M^2}\Bigl(N^{-1}\sum_{s=N+1}^{2N}\|x_{a_s}+x_{a_{s+1}}+\cdots+x_{a_{s+k-1}}\|\Bigr)^2\cr
&\ge&
\frac{1}{M^2}\Bigl\|N^{-1}\sum_{s=N+1}^{2N}(x_{a_s}+x_{a_{s+1}}+\cdots+x_{a_{s+k-1}})\Bigr\|^2.
\end{eqnarray*}
Next
\begin{eqnarray*}
\lefteqn{
\sum_{s=N+1}^{2N}\bigl(x_{a_s}+x_{a_{s+1}}+\cdots+x_{a_{s+k-1}}\bigr)
} \hspace{10mm}  \\
&=&
\sum_{s=N+1}^{N+k-1} (s-N)x_{a_s} +\sum_{s=N+k}^{2N} kx_{a_s}+\sum_{s=2N+1}^{2N+k-1}(2N+k-s)x_{a_s}.
\end{eqnarray*}
Hence
\[
A
\ge \frac{1}{M^2N^2}\Bigl(k\Bigl\|\sum_{s=N+1}^{2N}x_{a_s}\Bigr\|-k^2M\Bigr)^2
>
 \frac{1}{M^2}\Bigl(k\eta-\frac{k^2M}{N}\Bigr)^2
\ge \Bigl(\frac{k\eta}{2M}\Bigr)^2
\]
since $N\ge N_0\ge\frac{2kM}{\eta}$.

Next we estimate $A$ from above.
Using the inner product on $H$ we write
\[
A=
\frac{1}{a_NN}\sum_{r=1}^{a_N}\sum_{s=N+1}^{2N}\sum_{j,j'=0}^{k-1}
    \langle x_{r+a_{s+j}-a_s},x_{r+a_{s+j'}-a_s}\rangle=D+\sum_{0\le j<j'\le k-1}C_{j,j'},
\]
where
\[
D=\frac{1}{a_NN}\sum_{r=1}^{a_N}\sum_{s=N+1}^{2N}\sum_{j=0}^{k-1}\|x_{r+a_{s+j}-a_s}\|^2\le kM^2
\]
and
\[
C_{j,j'}=
\frac{2}{a_NN} \RRe \sum_{r=1}^{a_N}\sum_{s=N+1}^{2N}
   \langle x_{r+a_{s+j}-a_s},x_{r+a_{s+j'}-a_s}\rangle.
\]
Fix $j,j' \in \{ 0,\ldots,k-1 \} $ with $j < j'$.
Let 
\[
\BB=\Bigl\{m\in\NN: 1+\min_{s\in[N+1,2N]}\{a_{s+j}-a_{s}\}\le m\le a_N+\max_{s\in[N+1,2N]}\{a_{s+j}-a_{s}\}\Bigr\}.
\]
For all $m\in\BB$ let 
\begin{eqnarray*}
\cA_m
&=& \Bigl\{ s \in \{ N+1,\ldots,2N \} : \mbox{there exists an } r \in \{ 1,\ldots,a_N \}
      \mbox{ such that } m=r+a_{s+j}-a_s\Bigr\}  \\
&=& \Bigl\{ s \in \{N+1,\ldots,2N \} :  1\le m-a_{s+j}+a_s\le a_N\Bigr\}  \\
& = & \Bigl\{ s \in \{ N+1,\ldots,2N \} : 1+a_{s+j}-a_s\le m\le a_N+a_{s+j}-a_s\Bigr\}.
\end{eqnarray*}

Let $\BB_0=\{m: \max_s\{a_{s+j}-a_s\}< m\le a_N\}$. (Here and below in the proof we mean by $\max_s$ the maximum over $s\in\{N+1,\ldots,2N\}$). Note that $\cA_m=\{N+1,\dots,2N\}$ for all $m\in\BB_0$ and 
$$
\card(\BB\setminus\BB_0)\le 2\max_s(a_{s+j}-a_s)\le 
2\max_s(a_{s+k}-a_s)\le \max_s\frac{a_s}{2kM}\le\frac{a_{2N}}{2kM}\le\frac{ca_N}{2kM}.
$$

Then
\begin{eqnarray*}
|C_{j,j'}| 
&\le&
\frac{2}{a_NN}\Bigl|\sum_{m\in\BB}\Bigl\langle x_m, \sum_{s\in\cA_m} 
   x_{m+a_{s+j'}-a_{s+j}}\Bigr\rangle\Bigr|  \\
&\le&
\frac{2M}{a_NN}\sum_{m\in\BB}\Bigl\|\sum_{s\in\cA_m} x_{m+a_{s+j'}-a_{s+j}}\Bigr\|  \\
&\le&
\frac{2M^2}{a_NN}\sum_{m\in\BB}\Bigl\|\sum_{s\in\cA_m} x_{a_{s+j'}-a_{s+j}}\Bigr\| \\
&\le&\frac{2M^2}{a_NN}\sum_{m\in\BB_0}\Bigl\|\sum_{s=N+1}^{2N}x_{a_{s+j'}-a_{s+j}}\Bigr\|
+\frac{2M^2}{a_NN}\sum_{m\in\BB\setminus\BB_0}\Bigl\|\sum_{s\in\cA_m} x_{a_{s+j'}-a_{s+j}}\Bigr\| \\
&\le&\frac{2M^2}{a_NN}\card\BB_0 \Bigl\|\sum_{s=N+1}^{2N} x_{a_{s+j'}-a_{s+j}}\Bigr\|+
\frac{2M^2}{a_NN}\card(\BB\setminus\BB_0)\cdot MN \\
&\le&
\frac{2M^2}{N}\Bigl\|\sum_{s=N+1}^{2N} x_{a_{s+j'}-a_{s+j}}\Bigr\|+
M^2ck^{-1}.
\end{eqnarray*}
We have
\[
\sum_{s=N+1}^{2N}x_{a_{s+j'}-a_{s+j}}
=
\sum_{s=1}^{2N+j}x_{a_{s+j'-j}-a_s}-
\sum_{s=1}^{N+j}x_{a_{s+j'-j}-a_s},
\]
and so by (\ref{eSvdc1;1}) one has
\[
\Bigl\|\sum_{s=N+1}^{2N}x_{a_{s+j'}-a_{s+j}}\Bigr\|
\le k^{-1}(2N+j)+k^{-1}(N+j)\le 3Nk^{-1}+2.
\]
Hence
\[
|C_{j,j'}|
\le 6M^2k^{-1}+\frac{4M^2}{N}+M^2ck^{-1}\le 8M^2ck^{-1}
\]
and we deduce the upper bound 
\[
A
\le D+\sum_{0\le j<j'\le k-1}|C_{j,j'}|\le
kM^2+{k\choose 2}\cdot 8M^2ck^{-1}\le k M^2+4(k+1)cM^2=6kcM^2.
\]
Since 
\[
6kcM^2<\Bigl(\frac{k\eta}{2M}\Bigr)^2,
\]
we have a contradiction.
\end{proof}

\begin{definition}
We say that a subset $A\subset\NN$ is \emph{regular} if $\dens A>0$ and, for all $K\in\NN$ and each subset $B\subset [0,K]$, the set 
\begin{equation}\label{eq:regular-second-cond}
\bigl\{n\in\NN: n+j\in A\Leftrightarrow j\in B\quad(j=0,\dots,K)\bigr\}
\end{equation}
has density (either positive or equal to $0$).
\end{definition}
\noindent The second condition means -- if we identify $A$ with an infinite $0$-$1$-word -- that  every finite $0$-$1$-word appears in $A$ regularly in the sense that the beginnings of its appearances form a set which has density. Note that regularity is an asymptotic property, i.e., changing finitely many elements of $A$ does not change it.

\begin{example}
\begin{itemize}
\item[(a)] Both $\NN$ and all eventually periodic subsets of $\NN$ (e.g., infinite arithmetic progressions) are regular.
\item[(b)] As in Example \ref{ex:rtseq-density}, a large class of examples of regular sets comes from ergodic theory. Let $(X,\mu)$ be a probability space, $T:X\to X$ be a measure preserving transformation and $C\subset X$  with $\mu(C)>0$. For $x\in X$ consider the set  
$$
A:=\{n:\, T^nx\in C\}
$$ 
of return times to $C$. As discussed in Example \ref{ex:rtseq-density}, $A$ has density $\mu(C)$ for almost every $x\in X$. To verify the second property in the definition of regularity, let $K\in \NN$ and $B\subset [0,K]$. We see that the set  (\ref{eq:regular-second-cond}) equals 
\begin{eqnarray*}
&\ &\{n\in\NN:\, T^{n+j}x\in C \ \forall {j\in B},\ T^{n+j}x\notin C \ \forall {j\in [0,K]\setminus B}\}\\
&\ & \quad =
\{n\in\NN:\, T^{n}x\in \bigcap_{j\in B}T^{-j}C\cap \bigcap_{j\in [0,K]\setminus B} T^{-j}(X\setminus C)\}=\{n\in\NN:\, T^{n}x\in D\}
\end{eqnarray*}
for $$D:=\bigcap_{j\in B}T^{-j}C\cap \bigcap_{j\in [0,K]\setminus B} T^{-j}(X\setminus C).$$
Therefore the set  (\ref{eq:regular-second-cond}) is a return times set and hence, by Birkhoff's ergodic theorem, for almost every $x\in X$ has density $\mu(D)$ (which might be zero). Since there are countably many finite sets $B$, we obtain that for almost every $x\in X$ the set $A$ of return times to $C$ is regular. Moreover, as in Example \ref{ex:rtseq-density}, every $x\in X$ does it for uniquely ergodic systems.   As mentioned above, changing finitely many elements of $A$ does not change the regularity property, and we obtain the examples in (a) as a special case by taking $X$ being a finite set, $T$ being a rotation on $X$ and $\mu$ being the  normalized counting measure.
\item[(c)] A different class of examples are sets whose characteristic function is a normal $0$-$1$-sequence. It is well known that almost every $0$-$1$-sequence is normal.
\end{itemize}
\end{example}

\smallskip

Let $A\subset\NN$ be regular. Clearly then the set
$$
A_{k,m}=\bigl\{n\in A: n+m\in A, \card(A\cap[n,n+m])=k+1\bigr\}
$$
has density for all $k,m\in\NN, m\ge k$.

\begin{theorem}[Subsequential ergodic theorem for regular sets]\label{thm:subseq}
Let $T\in B(H)$ be power bounded, $\si_p(T)\cap\TT=\emptyset$, $A\subset\NN$ regular, $f\in P_m$ for some $m\in\NN$, and let $(h_n)$ be a bounded integer-valued sequence. Then
$$
(SOT)-\lim_{N\to\infty}\frac{1}{\card(A\cap[1,N])}\sum_{n\in A\cap[1,N]}T^{[f(n)]+h_n}=0.
$$
\end{theorem}

\begin{proof}
By induction on $m$.

For $m=1$ the statement was proved in Theorem \ref{caseP1}.

Let $m\ge 2$ and suppose the statement is true for $m-1$. Without loss of generality we may assume that $f$ is defined and increasing on $[0,\infty)$. 

For $n\in\NN$ let $g_A(n)$ be the $n$-th element of $A$, i.e., $g_A(n)\in A$ and $\card(A\cap[1, g_A(n)])=n$. Let $r=\sup\{|h_n|: n\in\NN\}$. Let $x\in H, \|x\|=1$. 

Define $\tilde f_n=[f(g_A(n))]+h_{g_A(n)}$. So we are supposed to show 
$$
\lim_{N\to\infty}\frac{1}{N}\sum_{n=1}^N T^{\tilde f_n}x=0.
$$

\medskip



\noindent{\bf Claim 1.}
$\lim_{n\to\infty}\frac{\tilde f_{n+1}-\tilde f_n}{\tilde f_n}=0.$  

\smallskip
\begin{proof}
Clearly $f(g_A(n))\ge f(n)\to\infty$ as $n\to\infty$. Since $|h_{g_A(n)}|\le r$, it is sufficient to show that
$$
\lim_{n\to\infty}\frac{f({g_A(n+1)})-f(g_A(n))}{f(g_A(n))}=0.
$$
Let $d=\dens A$ and $\e\in (0,d/2)$. For $n$ large enough we have
$$
n(d-\e)\le\card(A\cap[1,n])\le n(d+\e).
$$
So
\begin{equation}
\frac{n}{d+\e}\le g_A(n)\le \frac{n}{d-\e}.\label{gAn}
\end{equation}
Thus
$$
\frac{f({g_A(n+1)})-f(g_A(n))}{f(g_A(n))}\le
\frac{f(\frac{n+1}{d-\e})-f(\frac{n}{d+\e})}{f(\frac{n}{d+\e})}
$$
$$
\le
\frac{\bigl(\frac{n+1}{d-\e}-\frac{n}{d+\e}\bigr)f'(\xi)}{f(\frac{n}{d+\e})}
$$
for some $\xi\in(\frac{n}{d+\e},\frac{n+1}{d-\e})$.
We have
$$
\frac{n+1}{d-\e}-\frac{n}{d+\e}\le
\frac{(n+1)(d+\e)-n(d-\e)}{d^2-\e^2}\le
\frac{2\e n+d+\e}{d^2-\e^2}\le \frac{6\e n}{d^2}
$$
for $n$ large enough.
So
$$
\frac{f({g_A(n+1)})-f(g_A(n))}{f(g_A(n))}\le
\frac{6\e n}{d^2}\cdot\frac{f'(\xi)}{f(\frac{n}{d+\e})}
$$
$$
\le
\e\cdot\frac{6}{d^2}\cdot\frac{\xi f'(\xi)}{f(\xi)}\cdot
\frac{f(\xi)}{f(\frac{n}{d+\xi})}\cdot
\frac{n}{\xi}\le
\e\frac{6c}{d^2}\cdot\frac{f(\frac{2n}{d+\xi})}{f(\frac{n}{d+\xi})}\cdot(d+\e)
\le\e\cdot\hbox{const},
$$
where $c=\sup\{\frac{tf'(t)}{f(t)}: t\ge \frac{n}{d+\e}\}$, see Lemma \ref{L2.8} (iv) and (iii).
Since $\e>0$ was arbitrary, we have $\lim_{n\to\infty}\frac{\tilde f_{n+1}-\tilde f(n)}{\tilde f(n)}=0$.
\end{proof}

\noindent{\bf Claim 2.}
$\sup_n\frac{\tilde f_{2n}}{\tilde f_n}<\infty$
\smallskip

\begin{proof}
Let $d=\dens A$ and recall that 
$\limsup_{n\to\infty}\frac{f(2n)}{f(n)}<\infty$ by Lemma \ref{L2.8} (iii). Let $\e\in(0,d/2)$.
By (\ref{gAn}), we have
$$
\limsup_{n\to\infty}\frac{\tilde f_{2n}}{\tilde f_n}\le
\limsup_{n\to\infty}\frac{f(\frac{2n}{d-\e})+r}{f(\frac{n}{d+\e})-r-1}
=\limsup_{n\to\infty}\frac{f(\frac{2n}{d-\e})}{f(\frac{n}{d+\e})}
\le
\limsup\frac{f(\frac{4n}{d+\e})}{f(\frac{n}{d+\e})}\le
\limsup_{t\to\infty}\frac{f(4t)}{f(t)}<\infty.
$$
\end{proof}


So by Theorem \ref{tvdc201} it is sufficient to show that $\lim_{N\to\infty}N^{-1}\sum_{n=1}^N T^{\tilde f_{n+k}-\tilde f_n}x=0$ for all $k\in\NN$.

Fix $k\in\NN$. For each $m\ge k$ let $A_{k,m}=\bigl\{n\in A: n+m\in A, \card(A\cap[n,n+m])=k+1\bigr\}$. By assumption, each set $A_{k,m}$ has density. 

Let $\e>0$. Then there exists $M_0\in\NN$ such that
$$
\dens\bigcup_{m\le M_0} A_{k,m}\ge\dens A-\e.
$$

\bigskip

So  it is sufficient to show that for each $m, k\le m\le M_0$ such that $\dens A_{k,m}>0$ we have
$$
\lim_{N\to\infty}\frac{1}{\card(A_{k,m}\cap[1,N])} \sum_{n\in A_{k,m}\cap[1,N]} 
T^{\tilde f_{n+k}-\tilde f_n}x=0.
$$

However, this is equal to
$$
\lim_{N\to\infty}\frac{1}{\card(A_{k,m}\cap[1,N])} \sum_{n\in A_{k,m}\cap[1,N]}  
T^{[f(n)+m)]-[f(n)]+h_{n+m}-h_{n}} x
$$
$$
=\lim_{N\to\infty}\frac{1}{\card(A_{k,m}\cap[1,N])} \sum_{n\in A_{k,m}\cap[1,N]}
T^{[f(n+m)-f(n)]+\tilde h_{n}}x,
$$
where $\tilde h_{n}=[f(n+m)]-[f(n)]-[f(n+m)-f(n)]+h_{n+m}-h_{n}$.
Since $\sup_n\tilde h_n<\infty$, the last limit is equal to $0$ by the induction assumption.
\end{proof}

\begin{remark}\label{rem:not-full-regularity}
One can weaken the regularity assumption on the set $A$ in Theorem \ref{thm:subseq}. 
In fact, it suffices if $A$ has positive density and each of the sets $A_{k,m}$ has density.
\end{remark}

\begin{remark}
One cannot drop the assumption $\sigma_p(T)\cap\TT=\emptyset$ in Theorems \ref{caseP1} and \ref{thm:subseq}  even for contractions. Indeed, taking, e.g., $T=-I$ one can easily make any convergent averages $\aveN (-1)^{k_n}$ into divergent ones by adding $1$ to $k_n$ for appropriate $n$'s and vice versa.
\end{remark}

A direct consequence of Theorem \ref{thm:subseq} is the following. 
\begin{corollary}[Subsequential ergodic theorem]
Let $T\in B(H)$ be power bounded with $\si_p(T)\cap\TT=\emptyset$ and $f\in P_m$ for some $m\in\NN$. Then
$$
(SOT)-\lim_{N\to\infty}\frac{1}{N}\sum_{n=1}^N T^{[f(n)]}=0.
$$
\end{corollary}

\smallskip

We now turn our attention to weighted averages. For short we write $e(t):=e^{2\pi i t}$ for $t\in\RR$.

\begin{definition}
Let $g\in B$. We say that $g$ satisfies \emph{property $(Q)$} if the set
$$
\{n\in\NN: e(g(n))\in I\}
$$
is regular for each interval $I\subset\TT$.
\end{definition}

For examples of such $g$ see Section \ref{sec:hardy} below.

\begin{theorem}[Weighted subsequential ergodic theorem]\label{thm:main-weighted}
Let $T$ be a power bounded operator acting on a Hilbert space $H$, $\si_p(T)\cap\TT=\emptyset$, let $g$ satisfy $(Q)$ and let $f\in P_m$ for some $m\in\NN$.
Then
$$
(SOT)-\lim_{N\to\infty}\frac{1}{N}\sum_{n=1}^N e(g(n))T^{[f(n)]}=0.
$$
\end{theorem}

\begin{proof}
Let $k_0\in\NN$. For $k=1,\dots,k_0$ let $I_k=\bigl\{e(s):\frac{k-1}{k_0}\le s<\frac{k}{k_0}\bigr\}$. Then the sets $I_k$ are mutually disjoint and $\bigcup_{k=1}^{k_0}I_k=\TT$. Let $\la_k=e(\frac{2k-1}{2k_0})$. So $|\la_k-\la|\le\frac{\pi}{k_0}$ for each $\la\in I_k$.

Let $A_k:=\{n\in\NN: e(g(n))\in I_k\}$. By the definition, $A_k$ is regular for each $k$.

Let $x\in H$ be a unit vector. We have  for $M:=\sup_{n\in \NN_0}\|T^n\|$ by Theorem \ref{thm:subseq}
$$
\limsup_{N\to\infty}\frac{1}{N}\left\|\sum_{n=1}^N e(g(n))T^{[f(n)]}x\right\|\le
\limsup_{N\to\infty}\frac{1}{N}\sum_{k=1}^{k_0}
\Bigl\|\sum_{n\in A_k\cap[1,N]} e(g(n))T^{[f(n)]}x\Bigr\|
$$
$$
\le
\limsup_{N\to\infty}\frac{1}{N}\sum_{k=1}^{k_0}\Bigl(
\Bigl\|\sum_{n\in A_k\cap[1,N]} \la_k T^{[f(n)]}x\Bigr\|+
\Bigl\|\sum_{n\in A_k\cap[1,N]} \bigl(e(g(n))-\la_k\bigr)T^{[f(n)]}x\Bigr\|\Bigr)
$$
$$
\le 
\lim_{N\to\infty}
\sum_{k=1}^{k_0}
\frac{\card(A_k\cap[1,N])}{N}\Bigl\|\frac{1}{\card(A_k\cap[1,N])}\sum_{n\in A_k\cap[1,N]}T^{[f(n)]}x\Bigr\|+\frac{\pi M}{k_0}=
\frac{\pi M}{k_0}.
$$
Since $k_0\in\NN$ was arbitrary, we have
$$
\lim_{N\to\infty}\frac{1}{N}\sum_{n=1}^N e(g(n))T^{[f(n)]}x=0.
$$
\end{proof}
\begin{remark}
It again suffices if $g$ satisfies a weaker property than (Q), namely if for every rational (or dyadic) interval $I\subset \TT$ the sets $\{n\in\NN: e(g(n))\in I\}$ satisfy the property from Remark \ref{rem:not-full-regularity}.
\end{remark}

\section{Examples: Hardy functions\label{sec:hardy}}

Condition (Q) and classes $P_m$ are closely connected with Hardy functions.

Clearly $B$ with the natural algebraic operations is a ring. A subfield of $B$ is called \emph{Hardy field} if it is closed under differentiation. Denote by $U$ the union of all Hardy fields.

We summarize the basic properties of the set $U$.

\begin{theorem}
\begin{itemize}

\item[(i)] $U$ contains the class $L$ of logarithmico-exponential functions introduced by G. Hardy (i.e., all functions defined for all $t$ sufficiently large by a finite combination of ordinary algebraic operations ($+,-,\cdot,:$), powers, logarithms and exponential function. More precisely, $L$ is the smallest set containing the real constant functions, function $t\mapsto t$, and if $f,g\in L$ then $f+g,f-g,fg,f/g,\ln f,\exp f\in L$ (whenever the expression has sense).

\item[(ii)] If $f\in U$ then $f$ has continuous derivatives of all orders, which also belong to $U$.

\item[(iii)] If $f\in U$ is non-zero (in the sense of $B$) then either $f(t)>0$ or $f(t)<0$ for all $t$ sufficiently large. Similarly, if $f\in U$ is not constant, then either $f$ is increasing, or decreasing for all $t$ sufficiently large (since the derivative $f'$ is either positive, or negative). Consequently
the limit $\lim_{t\to\infty}f(t)\in\RR\cup\{\pm\infty\}$ exists for each $f\in U$.
\end{itemize}
\end{theorem}

Denote by $U_+$ the set of all functions $f\in U$ which are positive (for all $t$ sufficiently large).

For $f,g\in U_+$ we write $f\prec g$ if $\lim_{t\to\infty}\frac{f(t)}{g(t)}=0$ and $f\sim g$ if $\lim_{t\to\infty}\frac{f(t)}{g(t)}\in (0,\infty)$. We write $f\,{\precsim}\,\,g$ if either $f\prec g$ or $f\sim g$.

If $f,g\in U$ then they do not necessarily belong to the same Hardy field, so in general they are not comparable. However, if $f\in U$, $g\in L$ and $g\ne 0$ then $f/g\in U$ and so the limit $\lim_{t\to\infty} \frac{f(t)}{g(t)}$ exists. In particular, this is true for the function $g(t)=t^\al$ for each real $\al$ and we have the following by the L'Hospital rule.

\begin{proposition}
Let $f\in U_+$ and $\al\in\RR$. Then either $f\prec t^\al$, or $f\sim t^\al$, or $t^\al\prec f$. Moreover, if $f'>0$ and $\al>0$ then $f\prec t^\al\Rightarrow f'\prec t^{\al-1}$ and analogous implications hold if $f\sim t^\al$ or $t^\al\prec f$.
\end{proposition}

\begin{definition}
For $m\in\NN$ define the class $P'_m$ of functions $f\in U_+$ satisfying

(i) $t^{m-1}\,\prec f\,{\precsim}\,\, t^m$,

(ii) $f^{(m-1)}\,{\precsim}\,\,t f^{(m)}$.
\end{definition}

It is easy to see that $P'_m\subset P_m$.

\begin{example}  
Functions of the form $t^\alpha \ln^\beta t (\ln\ln  t)^\gamma$ ($m-1<\alpha< m, \beta,\gamma\in\RR$) 
or $\sum_{j=0}^k c_jt^{\al_j}$  ($c_0,\dots,c_k,\al_0,\dots,\al_k\in\RR$, $c_0>0$, $\al_0>\max\{0,\al_1,\dots,\al_k\}$, $m-1<\al_0\leq m$) are in $P'_m$ and therefore in $P_m$, 
which includes real polynomials of degree $m$ with positive leading coefficient. 
On the other hand, functions of the form $t^k\ln t$, $k\in\NN\cup\{0\}$ are not in $P_m$ and hence not in $P'_m$ for any $m$.  It would be interesting to know whether Theorem \ref{thm:subseq} still holds for 
these functions.
\end{example}

\medskip 

Let $g\in B$ be a function. The following conditions are sufficient for $g$ to satisfy property (Q):
\begin{itemize}
\item[(Q1)] For every interval $I\subset \TT$, the limit 
$$
\limaveN 1_{I}(e(g(n)))
$$
exists and is positive (where $1_I$ denotes the characteristic function of the interval $I$).
\item[(Q2)] For every $k\in\NN$ and every intervals $I_0,\ldots,I_k\subset \TT$, the limit
$$
\limaveN 1_{I_0}(e(g(n)))1_{I_1}(e(g(n+1)))\cdots 1_{I_k}(e(g(n+k)))
$$
exists.
\end{itemize}
Indeed, if (Q1) and (Q2) are satisfied, then to verify (Q) just take  $I_j$ in (Q2) to be either $I$ or $\TT\setminus I$ for appropriate $j$. Note also that (Q1) is necessary for (Q).

We first observe that (Q1) is satisfied if the sequence $(g(n))_{n=1}^\infty$ is equidistributed modulo $1$ or, equivalently, $(e(g(n)))_{n=1}^\infty$ is equidistributed in $\TT$. Recall that a sequence $(a_n)\subset \TT$ is called \emph{equidistributed} (or \emph{uniformly distributed}) \emph{in $\TT$} if for every interval $I\subset\TT$
$$
\lim_{N\to\infty} \frac{\card(n\in\{1,\ldots,N\}:\, a_n\in I)}N = \text{length}(I).
$$
Equidistribution of $(e(g(n)))_{n=1}^\infty$ in $\TT$  even occurs to be equivalent to (Q1) for subpolynomial $g\in U$, see Remark \ref{rem:Q1=equid} below.
Moreover, we have the following characterization of (Q2) in the spirit of Weyl's equidistribution criterion, see \cite[Thm. 2.1]{KN}.

\begin{proposition}\label{prop:equiv-Q}
Let $g\in B$. Then
the following assertions are equivalent.
\begin{itemize}
\item[(Q2')] For every $k\in\NN$ and every $f_0,\ldots,f_k\in C(\TT)$ the limit
$$
\limaveN f_0(e(g(n)))f_1(e(g(n+1)))\cdots f_k(e(g(n+k)))
$$
exists.
\item[(Q2'')] For every $k\in\NN$ and every  $m_0,\ldots,m_k\in\ZZ$ the limit
$$
\limaveN e(m_0g(n)+m_1g(n+1)+\cdots + m_k g(n+k))
$$
exists.
\end{itemize}
Moreover, if $(e(g(n)))_{n=1}^\infty$ is equidistributed in $\TT$, then (Q2)$\Leftrightarrow$(Q2')$\Leftrightarrow$(Q2'').
\end{proposition}
\begin{proof}
For $m\in\ZZ$ denote by $e_m:\TT\to\TT$ the function defined by $e_m(z):=z^m$.

(Q2')$\Rightarrow$(Q2'') follows by taking  $f_j:=e_{m_j}$. 

(Q2'')$\Rightarrow$(Q2') 
Let $f_0,\ldots,f_k\in C(\TT)$ and let $\e>0$. 
By the Weierstrass approximation theorem for trigonometric polynomials there exist functions $h_0,\ldots,h_k:\TT\to\CC$ which are linear combinations of $e_m$, $m\in \ZZ$, with $\|f_j-h_j\|_\infty<\e$ for every $j=0,\ldots,k$. Moreover, we can assume without loss of generality $\|f_j\|_\infty\leq 1$ and $\|h_j\|_\infty\leq 1$ for every $j=0,\ldots,k$.

By the triangle inequality we have for every $N,M\in\NN$
\begin{eqnarray*}
& &\left|\aveN f_0(e(g(n)))\cdots f_k(e(g(n+k))) - \aveM f_0(e(g(n)))\cdots f_k(e(g(n+k)))\right|\\
& &\ \leq 
\left|\aveN h_0(e(g(n)))\cdots h_k(e(g(n+k))) - \aveM h_0(e(g(n)))\cdots h_k(e(g(n+k)))\right|
\\
& &\qquad
+ 2(k+1)\e.
\end{eqnarray*}
Since the first term on the right hand side is by (Q2'') and linearity less than $\e$ for sufficiently large $N$ and $M$, (Q2') follows. 

\smallskip 
We now assume that $(e(g(n)))_{n=1}^\infty$ is equidistributed in $\TT$.

(Q2')$\Rightarrow$(Q2)  Let $I_0,\ldots,I_k\subset \TT$ be intervals and let $\e>0$. Let $f_j,h_j\in C(\TT)$  satisfy $0\leq f_j\leq 1_{I_j}\leq h_j\leq 1$ and $\int_\TT (h_j-f_j)<\e$ for every $j=0,\ldots,k$. We have by (Q2') and the triangle inequality
\begin{eqnarray*}
& &\limsup_{N\to\infty}\aveN 1_{I_0}(e(g(n)))\cdots 1_{I_k}(e(g(n+k))) -
\liminf_{N\to\infty}
\aveN 1_{I_0}(e(g(n)))\cdots 1_{I_k}(e(g(n+k)))\\
& &\qquad\leq
\limaveN
\left( h_0(e(g(n)))\cdots h_k(e(g(n+k))) - f_0(e(g(n)))\cdots f_k(e(g(n+k)))\right)\\
& & \qquad\leq
(k+1) \max_{j=0,\ldots,k} \limaveN (h_j-f_j)(e(g(n+j))).
\end{eqnarray*}
Since $(e(g(n+j)))_{n=1}^\infty$ is as well equidistributed in $\TT$ for every $j=0,\ldots,k$, by Weyl's equidistribution criterion the right hand side of the above is less than or equal to
$$
(k+1)\max_{j=0,\ldots,k} \int_\TT (h_j-f_j)<(k+1)\e.
$$ 
Since $\e>0$ was arbitrary, the averages
$$
\aveN 1_{I_0}(e(g(n)))\cdots 1_{I_k}(e(g(n+k)))
$$
converge. 

(Q2)$\Rightarrow$(Q2') follows analogously by 
approximating continuous functions $f_j$, $j=0,\ldots,k$, from above and below by linear combinations of characteristic functions of intervals. 
\end{proof}

\medskip

Thus, conditions (Q1) and (Q2'') imply property (Q) and, in fact, are equivalent to it. As preparation, we need the following characterization due to Boshernitzan \cite{B}, see also \cite[Remark 2.9]{EK}. 

\begin{theorem}[Properties of Hardy sequences \cite{B}]\label{thm:Bosh}
Let $g\in U$ be subpolynomial. Then the following assertions hold. 
\begin{itemize}
\item[(a)] The sequence $(e(g(n)))$ is equidistributed in $\TT$ if and only if
\begin{equation}\label{eq:hardy-equid}
\lim_{x\to+\infty} \frac{g(x)-p(x)}{\ln x}=\pm \infty
\end{equation}
for every polynomial $p$ with rational coefficients.
\item[(b)] The sequence $(e(g(n)))$ is dense in $\TT$ if and only if
\begin{equation}\label{eq:hardy-dens}
\lim_{x\to+\infty} (g(x)-p(x))=\pm \infty 
\end{equation}
for every polynomial $p$ with rational coefficients.
\item[(c)] The averages $\aveN e(g(n))$ converge if and only if either  (\ref{eq:hardy-equid}) holds or (\ref{eq:hardy-dens}) fails. Moreover, if (\ref{eq:hardy-equid}) holds, then $\limaveN e(g(n))=0$.
\end{itemize}
\end{theorem}
\begin{remark}[Rational polynomials] 
We see that property (Q1) (and hence (Q)) fails if $\lim_{x\to+\infty} (g(x)-p(x))$ is finite for some polynomial $p$ with rational coefficients. Note that for such $g$ the weighted ergodic averages 
$$
\aveN e(g(n))T^{[f(n)]+h_n}
$$
converge strongly for every $f\in P_m,m\in\NN,$ bounded $(h_n)\subset \ZZ$ and power bounded Hilbert space operator $T$ without unimodular eigenvalues by different reasons. Indeed, the sequence $(e(g(n)))$ is periodic in this case and convergence of the above weighted averages follows from Theorem \ref{thm:subseq} applied to infinite arithmetic progressions $A$ and the functions $f(a\cdot+b)\in P_m$ for suitable $a,b\in\NN$. 


The case of general polynomials $g$ is treated in Theorem \ref{thm:weighted-pol} and Corollary \ref{cor:pol} below.
\end{remark}

\begin{remark}\label{rem:Q1=equid}
As a corollary of Theorem \ref{thm:Bosh},  for subpolynomial $g\in U$ equidistribution of $(e(g(n)))$ in $\TT$ is equivalent to (Q1) and is  necessary for (Q). Indeed, if $(e(g(n)))$ is not equidistributed in $\TT$, then the limit $\limaveN e(g(n))$ either does not exist or equals to zero, both contradicting (Q1)  for $I=\TT$. Thus, by Proposition \ref{prop:equiv-Q}, (Q) is equivalent to the properties 
(\ref{eq:hardy-equid}) and (Q2'') for subpolynomial $g\in U$.
\end{remark}

\begin{example}
Consider $g\in U$ given by $g(x):=x\ln x$. Then $(e(g(n)))$ is equidistributed in $\TT$ by  Theorem \ref{thm:Bosh}(a), but the sequence $(e(\tilde{g}(n)))$ for $\tilde{g}(x):=g(x+1)-g(x)$ fails to converge in the Ces\`{a}ro sense by Theorem \ref{thm:Bosh}(c) since $g(x+1)-g(x)= g'(x)+o(1)=1+\ln x+o(1)$ by $g''=o(1)$.  
Therefore $g$ satisfies property (Q1) but fails to satisfy property (Q). Analogously, every function of the form $x\mapsto x^k\ln x$, $k\in\NN$, has the same property by considering the appropriate linear combination of $g(x),g(x+1),\ldots,g(x+k)$. Note that $x\mapsto \ln x$ satisfies neither (Q) nor (Q1).
\end{example}

We need the following simple property of subpolynomial Hardy functions. 
\begin{lemma}\label{lem:hardy-shift}
Let $g\in U$ be subpolynomial and consider $\tilde{g}$ given by $\tilde{g}(\cdot):=g(\cdot+1)$. Then $\tilde{g}=h+o(1)$ for some $h\in U$ belonging to the same Hardy field as $g$.
\end{lemma}
\begin{proof}
We can assume that $g\in U_+$ and let $k\in \NN$ be such that $g(t)\prec t^{k+1}$. Then 
$g^{(k+1)}=o(1)$, so 
for $t$ large enough 
$$
\tilde{g}(t)=g(t+1)= g(t)+g'(t)+\ldots+\frac{g^{(k)}(t)}{k!}  
+o(1)
$$
holds. The assertion follows.
\end{proof}

We now introduce the following classes of Hardy functions. For $l\in \NN_{0}$ denote 
$$
M_l:=\{g\in U_+:\ t^l\ln t\prec g\prec t^{l+1}\}.
$$
The following characterises Hardy functions satisfying (Q). Without loss of generality we consider (eventually) positive functions. 
\begin{theorem}[Property (Q) for Hardy functions]\label{thm:Q-hardy}
Let $g\in M_l$ for some $l\in \NN_{0}$. Then $g$ satisfies (Q).
\end{theorem}
\begin{proof}
%
%
%
%
%
Assume that $g\in M_l$ for some $l\in \NN_0$. Then $(e(g(n)))$ is equidistributed in $\TT$ by Theorem \ref{thm:Bosh}(a) and it remains to show (Q2'') by Proposition \ref{prop:equiv-Q}. Take $k\in \NN$ and $m_0,\ldots,m_k\in \ZZ$. Write for $s:=m_0+\ldots+m_k$
\begin{eqnarray*}
\tilde{g}(t)&:=&m_0g(t)+\ldots+m_kg(t+k)\\
&=&m_k(g(t+k)-g(t+k-1))+(m_k+m_{k-1})(g(t+k-1)-g(t+k-2))+\ldots\\
& &+(m_k+\ldots+m_1)(g(t+1)-g(t))+sg(t).
\end{eqnarray*}
Note that by the L'Hospital rule $t^{l-1}\ln t\prec g'\prec t^{l}$,$\ldots$,$ \ln t\prec g^{(l)}\prec t$, and $g^{(l+1)}=o(1)$. So 
we have 
\begin{eqnarray*}
\tilde{g}(t)&=&
m_k g'(t+k-1) + (m_k+m_{k-1})g'(t+k-2)+\ldots + (m_k+\ldots+m_1)g'(t) \\
&\ &+ m_k g''(t+k-1) + (m_k+m_{k-1})g''(t+k-2)+\ldots + (m_k+\ldots+m_1)g''(t)+\ldots \\
&\ &+m_k g^{(l)}(t+k-1) + (m_k+m_{k-1})g^{(l)}(t+k-2)+\ldots+ (m_k+\ldots+m_1)g^{(l)}(t)   \\
&\ &
+ o(1)+sg(t) \\
&=&: h(t)+ sg(t)+o(1).
\end{eqnarray*}
By Lemma \ref{lem:hardy-shift} we can assume without loss of generality that $h$ and $\tilde{g}$ are Hardy functions from the same Hardy field as $g$.

We now consider the following cases. 

Case 1: $s\neq 0$. Since $h\prec t^l$, we have $\tilde{g}\in M_l$ and therefore the averages $\aveN e(\tilde{g}(n))$ converge by the equidistribution property and Weyl's criterion. 

Case 2: $s= 0$. If $l=0$ then $g'=o(1)$ and the averages $\aveN e(\tilde{g}(n))$ clearly converge. If $l\geq 1$, then $g'\in M_{l-1}$,$\ldots,g^{(l)}\in M_0$ and $\tilde{g}$ is of the form 
\begin{eqnarray*}
\tilde{g}(t)&=&\tilde{m}_0g'(t)+\ldots + \tilde{m}_{k-1}g'(t+k-1)+\ldots   \\
&\ &+\tilde{m}_0 g^{(l)}(t)+  \ldots 
+ \tilde{m}_{k-1}g^{(l)}(t+k-1)+o(1).
\end{eqnarray*}
Using Lemma \ref{lem:hardy-shift}, by induction on $k$ and considering the two cases in every step we obtain that the averages $\aveN e(\tilde{g}(n))$ converge. Property (Q) follows, completing the proof. 
\end{proof}
The following shows in particular that the converse implication in Theorem \ref{thm:Q-hardy} for functions in $U_+$ does not hold in general. 
\begin{remark}
For most $g\in U_+\notin \bigcup_{l\in\NN\cup\{0\}}M_l$ property (Q) fails but sometimes it  holds.  There are several cases to consider.

Case 1: $0\leq g\precsim \ln t$.
Then $(e(g(n)))$ is not equidistributed in $\TT$ by Theorem \ref{thm:Bosh}(a) and (Q) fails. 

Case 2: There exists $l\in \NN$ with $t^l\prec g\precsim t^l\ln t$. Then by the L'Hospital rule $1\prec g^{(l)}\precsim \ln n$ and the averages $\aveN e(g^{(l)}(n))$ diverge by Theorem \ref{thm:Bosh}(c). It remains to find $m_0,\ldots, m_l$ so that $m_0g(t)+\ldots+m_lg(t+l) = g^{(l)}(t)+o(1)$. This is clearly possible by discrete approximation of the derivative(s) and using the fact that $g^{(l+1)}=o(1)$. Thus, (Q) fails.

Case 3: $g\sim t^l$ for some $l\in\NN$. Here, the situation is not homogeneous. For $g$ given by $g(t)=t^l$ or $g(t)=t^l+\ln t$,  $(e(g(n)))$ is not equidistributed in $\TT$ by Theorem \ref{thm:Bosh}(a) and (Q) fails. On the other hand, for $g$ given by $g(t)=t^l+(\ln t)^2$, property (Q) holds. Indeed, $(e(g(n)))$ is equidistributed in $\TT$ by Theorem \ref{thm:Bosh}(a) implying (Q1). Consider a linear combination $\tilde{g}$ of $g(\cdot)$,$g(\cdot + 1),\ldots$. As in the proof of Theorem \ref{thm:Q-hardy}, since $g^{(j)}=l(l-1)\cdots (l-j+1)t^{l-j}+o(1)$ for $j\in\{1,\ldots,l\}$, $\tilde{g}$ is up to $o(1)$ either a rational polynomial or a rational polynomial plus a constant times $\ln^2t$. In both cases, $(e(g(n)))$ is Ces\`aro convergent by 
 Theorem \ref{thm:Bosh} (c) 
implying (Q2'').
\end{remark}

Thus Theorems \ref{thm:main-weighted} and \ref{thm:Q-hardy} imply the following weighted ergodic theorem. 

\begin{corollary}\label{cor:weighted}
Let $T$ be a power bounded operator acting on a Hilbert space $H$, $\si_p(T)\cap\TT=\emptyset$, let $g\in M_l$ for some $l\in\NN_0$ and let $f\in P_m$ for some $m\in\NN$.
Then
$$
(SOT)-\lim_{N\to\infty}\frac{1}{N}\sum_{n=1}^N e^{2\pi i g(n)}T^{[f(n)]}=0.
$$
\end{corollary}

We finally consider
polynomial weights which were excluded in Corollary \ref{cor:weighted}. 

\begin{theorem}\label{thm:weighted-pol}
Let $T\in B(H)$ be power bounded with $\sigma_p(T)\cap \TT=\emptyset$, $f\in P_m$ for some $m\in \NN$. Then 
$$
(SOT)-\limaveN e(g(n))T^{[f(n)]}=0
$$
holds for every $g$ of the form $g(t)=\sum_{j=0}^kc_jt^{\alpha_j}$, $k\in\NN_0$, $c_0,\ldots,c_k, \alpha_0,\ldots,\al_k\in\RR$.
\end{theorem}
\begin{proof}
Without loss of generality we can assume that $c_0,\ldots,c_k\neq 0$ and $\alpha_0>\alpha_1>\ldots>\alpha_k$. Moreover, by Theorem \ref{thm:subseq} applied to $A=\NN$ we can assume that $\al_0\geq 0$. 

If $\alpha_0\notin \NN_0$, the assertion follows from Corollary \ref{cor:weighted}. So we can assume that $\alpha_0=l\in \NN_0$. We proceed by induction on $l$. The induction basis $l=0$ follows from Theorem \ref{thm:subseq} applied to $A=\NN$. Let now $l\in \NN$ and assume that the assertion holds for smaller powers.
There are two cases to consider.

Case 1: $c_0$ is rational. Since $(e(c_0n^l))$ is periodic, by going to arithmetic progressions (and again using that $f(a\cdot+b)\in P_m$ for all $a,b\in\NN$) we can assume without loss of generality that the term $c_0n^l$ is not there. If $\al_1\notin \NN_0$, the assertion follows from  Corollary \ref{cor:weighted}, otherwise it follows from the induction hypothesis. 

Case 2: $c_0$ is irrational. Then $(e(g(n)))$ is equidistributed by Theorem \ref{thm:Bosh} (a) implying (Q1). Let $\tilde{g}(\cdot):=m_0g(\cdot)+m_1g(\cdot+1)+\ldots+m_Kg(\cdot+K)$ for arbitrary $K\in\NN_0$ and $m_0,\ldots,m_K\in \ZZ$.  
If $m_0+\ldots+m_K\neq 0$, then $(e(\tilde{g}(n)))$ is equidistributed by Theorem \ref{thm:Bosh} (a) implying (Q2''), and the assertion follows from Corollary \ref{cor:weighted}. 
If $m_0+\ldots+m_K=0$, then, by the same argument as in the proof of  Theorem \ref{thm:Q-hardy},  the leading term of $\tilde{g}$ is (up to $o(1)$) a linear combination of the derivatives of $g$ which are of the same form as $g$ with powers decreased by $1$. Repeating the argument, Ces\`aro convergence of $(e(\tilde{g}(n)))$ easily follows by induction and Theorem \ref{thm:Bosh} (a). 

The assertion follows now from Theorem \ref{thm:main-weighted}.
\end{proof}

\begin{remark}
Using the same techniques, one can replace $g(t)=\sum_{j=0}^kc_jt^{\alpha_j}$ in Theorem \ref{thm:weighted-pol} by $g(t)=\sum_{j=0}^kc_j(t+b_j)^{\alpha_j}$ for $k\in\NN_0$, $c_0,\ldots,c_k, b_0,\ldots,b_k, \alpha_0,\ldots,\al_k\in\RR$. We leave the details to the reader. 
Moreover, for real polynomials the argument in the above proof simplifies due to Weyl's equidistribution theorem for polynomials or periodicity reasons, respectively. 
\end{remark}

In particular, we have the following generalization of the result of ter Elst, M\"uller \cite{tEM}.

\begin{corollary}[Convergence of polynomial averages with polynomial weights]\label{cor:pol}
Let $T\in B(H)$ be power bounded. Then the weighted averages 
\begin{equation}\label{eq:poly}
\aveN e(q(n))T^{[p(n)]}
\end{equation}
converge strongly for every real polynomials $p,q\in\RR[\cdot]$.
\end{corollary}

\begin{proof}
Let $H=H_1\oplus H_2$ be the Jacobs-Glicksberg-deLeeuw decomposition discussed in the introduction. Convergence of averages (\ref{eq:poly}) on $H_2$ follows directly from Theorem \ref{thm:weighted-pol}, so we can assume without loss of generality that $H=H_1$. By the standard approximation argument, strong convergence of (\ref{eq:poly})  follows from convergence of  (\ref{eq:poly}) for operators of the form $T=\lambda$, $\lambda\in \TT$. So it remains to show that the scalar averages 
\begin{equation}\label{eq:gen-poly}
\aveN e(q(n)+\alpha[p(n)])
\end{equation}
converge for every $\alpha \in \RR$. The function $t\mapsto q(t)+\alpha[p(t)]$ is a so-called generalized polynomial, and convergence of (\ref{eq:gen-poly}) follows from recent results in ergodic theory, see Bergelson, Leibman \cite[Corollary 0.26]{BL}.
\end{proof}


\begin{thebibliography}{10}

\bibitem[A]{A}
{\sc I. Assani}, Wiener Wintner Ergodic Theorems, World Scientific Publishing Co. Inc., River Edge, NJ,
2003.

\bibitem[BLRT]{BLRT}
{\sc D.~Berend,  M.~Lin, J.~Rosenblatt, 
A.~Tempelman}, Modulated
  and subsequential ergodic theorems in Hilbert and Banach spaces.
\newblock {\em Ergodic Theory Dynam. Systems} {\bf 22} (2002),  1653--1665.

\bibitem[BL]{BL}
{\sc V.~Bergelson, A.~Leibman},  Distribution of values of bounded generalized polynomials, \emph{Acta Math.} \textbf{198} (2007), 155--230.


\bibitem[BKS]{BKS}
{\sc V.~Bergelson,  G.~Kolesnik,  Y.~Son}, Uniform distribution of subpolynomial functions along primes and applications, \emph{J. Anal. Math.} \textbf{137} (2019), 135--187.

\bibitem[BE]{BE}
{\sc J.~Blum,  
B.~Eisenberg}, Generalized summing sequences and the
  mean ergodic theorem.
\newblock {\em Proc. Amer. Math. Soc.} {\bf 42} (1974),  423--429.

\bibitem[B]{B} {\sc M.~D.~Boshernitzan}, Uniform distribution and Hardy fields, \emph{J. Anal. Math.} \textbf{62} (1994), 225--240. 

\bibitem[BKQW]{BKQW}
{\sc M.~Boshernitzan, G.~Kolesnik, A.~Quas, 
M.~Wierdl}, Ergodic averaging sequences, \emph{J. Anal. Math.} \textbf{95} (2005), 63--103. 

\bibitem[BW]{BW}
{\sc M.~Boshernitzan, M.~Wierdl}, Ergodic theorems along sequences and Hardy fields, \emph{Proc. Natl. Acad. Sci. USA} \textbf{93} (1996), 8205--8207. 



\bibitem[Bou]{Bou}
{\sc J.~Bourgain}, Pointwise ergodic theorems for arithmetic sets, \emph{Inst.  Hautes \'Etudes Sci. Publ. Math.} (1989), no.~69, 5--45, With an appendix by the author, Harry Furstenberg, Yitzhak Katznelson, and Donald S. Ornstein.


\bibitem[BFKO]{BFKO}
{\sc J.~Bourgain, H.~Furstenberg, Y.~Katznelson, D.~S. Ornstein}, Appendix on return-time sequences, \emph{Inst. Hautes \'Etudes Sci. Publ. Math.} (1989),
  no.~69, 42--45.

\bibitem[BM]{BM} 
{\sc Z.~Buczolich, R.~D.~Mauldin}, Divergent square averages, \emph{Ann. of Math. (2)} \textbf{171} (2010),  1479--1530.


\bibitem[EW]{EW}
{\sc M.~Einsiedler, 
T.~Ward}, Ergodic Theory with a View  towards Number Theory.
\newblock Graduate Texts in Mathematics 259. Springer-Verlag, London, 2011.


\bibitem[E]{E}
{\sc T.~Eisner}, Stability of Operators and Operator Semigroups. Operator Theory: Advances and Applications, vol. 209, Birkh\"auser Verlag, Basel, 2010.


\bibitem[EFHN]{EFHN}
{\sc T.~Eisner, B.~Farkas, M.~Haase,  R.~Nagel}, Operator Theoretic Aspects of Ergodic Theory. Graduate Texts in Mathematics, Springer, 2015.




\bibitem[EK]{EK}
{\sc T.~Eisner, B.~Krause},  (Uniform) convergence of twisted ergodic averages, \emph{Ergodic Theory Dynam. Systems} \textbf{36} (2016), 2172--2202.

\bibitem[tEM]{tEM}
{\sc A.~F.~M.~ter Elst, V.~M\"uller}, A van der Corput-type lemma for power bounded operators, \emph{Math. Z.} \textbf{285} (2017), 143--158.

\bibitem[Fo]{Fo} {\sc S.~R. Foguel}, A counterexample to a problem of Sz.-Nagy, \emph{Proc. Amer. Math. Soc.} \textbf{15} (1964), 788--790.

\bibitem[F]{F}
{\sc N.~Frantzikinakis}, Multiple recurrence and convergence for Hardy sequences of polynomial growth, \emph{J. Anal. Math.} \textbf{112} (2010), 79--135. 

\bibitem[FW]{FW}
{\sc N.~Frantzikinakis, M.~Wierdl}, A Hardy field extension of Szemer\'edi's theorem, \emph{Adv. Math.} \textbf{222} (2009),  1--43.




\bibitem[GT]{GT} 
{\sc B.~Green, T.~Tao}, The M\"obius function is strongly orthogonal to nilsequences,
\emph{Ann. of Math. (2)} \textbf{175} (2012), 541--566.



\bibitem[H]{H}
{\sc P.~R. Halmos}, On Foguel's answer to Nagy's question, \emph{Proc. Amer. Math. Soc.} 
\textbf{15} (1964), 791--793.

\bibitem[JL1]{JL1}
{\sc L.~K.~Jones, M.~Lin}, 
Ergodic theorems of weak mixing type,
\emph{Proc. Amer. Math. Soc.} \textbf{57} (1976), 50--52.

\bibitem[JL2]{JL2}
{\sc L.~K.~Jones, M.~Lin},
Unimodular eigenvalues and weak mixing,
\emph{J. Functional Analysis} \textbf{35} (1980),  42--48.


\bibitem[KL]{KL}
{\sc I.~Kornfeld, M.~Lin}, Weak almost periodicity of $L_1$ contractions and coboundaries of non-singular transformations, \emph{Studia Math.} \textbf{138} (2000), 225--240.



\bibitem[KZ]{KZ}
{\sc B.~Krause, P.~Zorin-Kranich}, A random pointwise ergodic theorem with Hardy field weights, \emph{Illinois J. Math.} \textbf{59} (2015),  663--674.

\bibitem[KN]{KN}
{\sc L.~Kuipers, H.~Niederreiter}, Uniform Distribution of Sequences.
  Wiley-Interscience [John Wiley \& Sons], New York-London-Sydney, 1974, Pure  and Applied Mathematics.


\bibitem[L]{L} 
{\sc E.~Lesigne}, Spectre quasi-discret et th\'eor\`eme ergodique de Wiener-Wintner pour les polyn\^omes, \emph{Ergodic Theory Dynam. Systems}  \textbf{13}  (1993), 767--784.

\bibitem[LOT]{LOT} 
{\sc M.~Lin, J.~Olsen, A.~Tempelman},
On modulated ergodic theorems for Dunford-Schwartz operators,
\emph{Illinois J. Math.} \textbf{43} (1999),  542--567. 

\bibitem[N]{N}
{\sc R.~Nair}, On polynomials in primes and {J}. {B}ourgain's circle method approach to ergodic theorems. {II}, \emph{Studia Math.} \textbf{105} (1993),  207--233.

\bibitem[P]{P} {\sc G. Pisier}, A polynomially bounded operator on Hilbert space which is not similar to a contraction, \emph{J. Amer. Math. Soc.} \textbf{10} (1997), 351--369.

\bibitem[RW]{RW}
{\sc J.~M. Rosenblatt, M.~Wierdl}, Pointwise ergodic theorems via harmonic analysis, \emph{Ergodic theory and its connections with harmonic analysis} ({A}lexandria, 1993), London Math. Soc. Lecture Note Ser., vol. 205, Cambridge Univ. Press, Cambridge, 1995, pp.~3--151.

\bibitem[S]{S} 
{\sc P.~Sarnak}, Three lectures on the M\"obius Function randomness and dynamics, 2010, see
http://publications.ias.edu/sarnak/paper/506.

\bibitem[SzN]{SzN} {\sc B. Sz.-Nagy}, Completely continuous operators with uniformly bounded iterates, \emph{Magyar Tud. Akad. Mat. Kutato Int. K\"ozl.} \textbf{4} (1959), 89--93.

\bibitem[W]{W}
{\sc M.~Wierdl}, Pointwise ergodic theorem along the prime numbers, \emph{Israel J. Math.} \textbf{64} (1988), 315--336 (1989).












\end{thebibliography}
\end{document}